\documentclass[12pt,leqno]{amsart}
\usepackage{amsmath}
\usepackage{amsthm}
\usepackage{amssymb}
\def\overset#1#2{{\mathrel{\mathop {{#2}_{}}\limits^{#1}}}}
\def\underset#1#2{{\mathrel{\mathop {{}_{} {#2}}\limits_{{#1}_{}}}}}
\def\upplim_#1{\underset{#1}{\overline\lim}\;}
\def\lowlim_#1{\underset{#1}{\underline\lim}\;}
\newtheorem{corollary}[equation]{Corollary}
\newtheorem{claim}[equation]{\indent \rm {\it Claim}}
\newtheorem{definition}[equation]{Definition}

\newtheorem{lemma}[equation]{Lemma}

\newtheorem{proposition}[equation]{Proposition}

\newtheorem{theorem}[equation]{Theorem}

\newcommand{\C}{{\mathbb{C}}}

\newcommand{\ric}{\mathrm{Ric}}

\newcommand{\N}{\mathbb{N}}
\newcommand{\B}{\mathbb{B}}

\renewcommand{\P}{{\mathbb{P}}}

\newcommand{\R}{{\mathbb{R}}}

\newcommand{\rank}{\mathrm{rank}}
\newcommand{\supp}{\mathrm{Supp}\,}

\newcommand{\Z}{\mathbb{Z}}
\numberwithin{equation}{section}

\begin{document}
\title[Differentiably nondegenerate meromorphic mappings on K\"{a}hler manifolds]{Algebraic dependence and finiteness problems of differentiably nondegenerate meromorphic mappings on K\"{a}hler manifolds} 

\author{Si Duc Quang$^{1,2}$}
\address{$^{1}$Department of Mathematics, Hanoi National University of Education\\
136-Xuan Thuy, Cau Giay, Hanoi, Vienam.}
\address{$^{2}$Thang Long Institute of Mathematics and Applied Sciences\\
Nghiem Xuan Yem, Hoang Mai, Ha Noi.}
\email{quangsd@hnue.edu.vn}

\def\thefootnote{\empty}

\begin{abstract} {Let $M$ be a complete K\"{a}hler manifold, whose universal covering is biholomorphic to a ball $\B^m(R_0)$ in $\C^m$ ($0<R_0\le +\infty$). Our first aim in this paper is to study the algebraic dependence problem of differentiably meromorphic mappings. We will show that if $k$ differentibility nondegenerate meromorphic mappings $f^1,\ldots,f^k$ of $M$ into $\P^n(\C)\ (n\ge 2)$ satisfying the condition $(C_\rho)$  and sharing few hyperplanes in subgeneral position regardless of multiplicity then $f^1\wedge\cdots\wedge f^k\equiv 0$. For the second aim, we will show that there are at most two different differentiably nondegenerate meromorphic mappings of $M$ into $\P^n(\C)$ sharing $q\ (q\sim 2N-n+3+O(\rho))$ hyperplanes in $N-$subgeneral position regardless of multiplicity. Our results generalize previous finiteness and uniqueness theorems for differentiably meromorphic mappings of $\C^m$ and extend some previous results for the case of mappings on K\"{a}hler manifold.}
\end{abstract}

\maketitle

\section{Introduction}

In \cite{Fu86}, Fujimoto proved the following theorem, which is the first uniqueness theorem for meromorphic mappings on K\"{a}hler manifold.
 
\vskip0.2cm
{\bf Theorem C} (see \cite[Main Theorem]{Fu86}).
{\it Let $M$ be an $m$-dimensional connected K\"{a}hler manifold whose universal covering is biholomorphic to a ball $\B^m(R_0)$ in $\C^m$\ ($0<R_0\le +\infty$), and let $f,g$ be a linearly non-degenerate meromorphic mappings of $M$ into $\P^n(\C)\ (m\ge n)$ satisfying the condition $(C_\rho)$ for a positive number $\rho$. Let $H_1,\ldots,H_q$ be $q$ hyperplanes of $\P^n(\C)$ in general possition. Assume that

i) $f=g$ on $\bigcup_{i=1}^q\left (f^{-1}(H_i)\cup g^{-1}(H_i)\right)$,

ii) If $q> n+1+2\rho (l_f+l_g)+m_f+m_g$. 

Then $f=g$.}

Here, we say that $f$ satisfies the condition ($C_\rho$) if there exists a nonzero bounded continuous real-valued function $h$ on $M$ such that
$$\rho\Omega_f+dd^c\log h^2\ge \ric\ \omega,$$
and the numbers $l_f,l_g,m_f,m_g$ are positive numbers estimated in an explicit way. For the case where $f$ and $g$ are differentiably non-degenetate, we can take $m_f=m_g=1$ and $l_f=l_g=n$.

Our first purpose in this paper is to extend the above theorem to the case where $k$ differentiably nondegenerate meromorphic mappings $f^1,\ldots,f^k\ (2\le k\le n+1)$ sharing a the family of hyperplanes in $N-$subgeneral position. To state our result, we need to recall some following.

Let $M$ be an $m$-dimensional connected K\"{a}hler manifold whose universal covering is biholomorphic to a ball $\B^m(R_0)$ in $\C^m$\ ($0<R_0\le +\infty$). Let $f$ be a non-constant meromorphic mapping of $\B^m(R_0)$ into $\P^n(\C)$ with a reduced representation $f = (f_0 : \dots : f_n)$, and $H$ be a hyperplane in $\P^n(\C)$ given by $H=\{a_0\omega_0+\cdots +a_n\omega_n=0\},$ where $(a_0,\ldots ,a_n)\ne (0,\ldots ,0)$.
Set  $(f,H)=\sum_{i=0}^na_if_i$. We see that $\nu_{(f,H)}$ is the pull-back divisor of $H$ by $f$ and is also the divisor generated by the function $(f,H_i)$.

Let $H_1,\ldots ,H_q$ be $q$ hyperplanes  of $\P^n(\C)$ in $N-$subgeneral position.
Let $d$ be a positive integer, $\rho$ be a positive number and $f$ be a differentiably nondegenerate meromorphic mapping from $M$ into $\P^n(\C)\ (m\ge n)$ satisfying the condition $(C_\rho)$. We consider the set $\mathcal{D}(f,\{H_i\}_{i=1}^q,\rho,d)$ of all differentiably nondegenerate meromorphic mappings 
$g$ from $M$ into $\P^n(\C)$ satisfying the condition $(C_\rho)$ and the following conditions:
\begin{itemize}
\item[(a)] $\nu_ {(f,H_i)}^{[d]}=\nu_{(g,H_i)}^{[d]}\quad (1\le i \le q),$
\item[(b)] $f(z) = g(z)$ on $\bigcup_{i=1}^{q}f^{-1}(H_i)$.
\end{itemize}
Here, $\nu^{[d]}=\min\{\nu,d\}$ for each divisor $\nu$.

Then, our first result in this paper is stated as follows.

\begin{theorem}\label{thm1.1}
Let $M$ be an $m$-dimensional connected K\"{a}hler manifold whose universal covering is biholomorphic to a ball $\B^m(R_0)$ in $\C^m$\ ($0<R_0\le +\infty$), and let $f$ be a differentiably nondegenerate meromorphic mapping of $M$ into $\P^n(\C)\ (m\ge n)$ satisfying the condition $(C_\rho)$ for a positive number $\rho$. Let $H_1,\ldots,H_q$ be $q$ hyperplanes of $\P^n(\C)$ in $N-$subgeneral possition. Let $f^1,\ldots,f^k\ (2\le k\le n+1)$ be elements in $\mathcal{D}(f,\{H_i\}_{i=1}^{q},\rho,1)$.

a) If $q> 2N-n+1+\dfrac{k(2N-n+1)}{(k-1)(n+1)}+kn\rho$ then $f^1\wedge\cdots\wedge f^k\equiv 0$.

b) If $\dim f^{-1}(H_i)\cap f^{-1}(H_j) \le m-2 \quad (1 \le i<j \le q)$ and 
$$q> 2N-n+1+\frac{kn(2N-n+1)}{(k-1)N(n+1)}+\frac{kn^2\rho}{N}$$
 then $f^1\wedge\cdots\wedge f^k\equiv 0$.
\end{theorem}
Letting $k=2$, we immediately get the following uniqueness theorem.
\begin{corollary}
Let $M,f,H_i (1\le i\le q),\rho$ be as in Theorem \ref{thm1.1}. 

a) If $q> 2N-n+1+\dfrac{2(2N-n+1)}{(n+1)}+2n\rho$ then $\sharp\mathcal{D}(f,\{H_i\}_{i=1}^{q},\rho,1)=1$.

b) If $\dim f^{-1}(H_i)\cap f^{-1}(H_j) \le m-2 \quad (1 \le i<j \le q)$ and
$$q> 2N-n+1+\frac{2n(2N-n+1)}{N(n+1)}+\frac{2n^2\rho}{N}$$
 then $\sharp\mathcal{D}(f,\{H_i\}_{i=1}^{q},1)=1$.
\end{corollary}
Here, by $\sharp S$ we denote the cardinality of the set $S$.

Remark: Suppose that $q=n+4$ and $\{H_i\}_{i=1}^{n+4}$ is in general position, i.e., $N=n$. Then
the assumption of the above corollary is fulfilled with $\rho<\frac{1}{2n}$. Then this result is an extension of the uniqueness theorem for differentiably non-degenerate meromorphic mappings into $\P^n(\C)$ sharing a normal crossing divisor of degree $n+4$ given firstly by Drouilhet \cite[Theorem 4.2]{D}.

We would like to emphasize here that, in order to study the finiteness problem of meromorphic mappings for the case of mappings from $\C^m,$  almost all authors use Cartan's auxialiary functions (see Definition \ref{2.2}) and compare the counting functions of these auxialiary functions with the characteristic functions of the mappings. However, in the general case of K\"{a}hler manifold, this method may do not work since this comparation does not make sense if the growth of the characteristic functions do not increase quickly enough. In order to overcome this difficulty, in \cite{Q20}, we introduced the notions of ``functions of small integration'' and ``functions of bounded integration''. Using this notions, we will extend the finiteness theorems for differentiably non-degenerate meromorphic mappings of $\C^m$ into $\P^n(\C)$ sharing $n+3$ hyperplanes (see \cite{Q14}) to the case of K\"{a}hler manifolds. Our last result is stated as follows.
\begin{theorem}\label{thm3}
Let $M$ be an $m$-dimensional connected K\"{a}hler manifold whose universal covering is biholomorphic to a ball $\B^m(R_0)$ in $\C^m$\ ($0<R_0\le +\infty$), and let $f$ be a differentiably non-degenerate meromorphic mapping of $M$ into $\P^n(\C)\ (m\ge n)$ satisfying the condition $(C_\rho)$ for a positive number $\rho$. Let $H_1,\ldots,H_q$ be $q$ hyperplanes of $\P^n(\C)$ in $N-$subgeneral possition such that
$$\dim f^{-1}(H_i)\cap f^{-1}(H_j) \le m-2 \quad (1 \le i<j \le q).$$ 
Assume that 
$$q>2N-n+1+\frac{9n(2N-n+1)}{5N(n+1)}+\rho\left (3n+\dfrac{9n}{5N}\right ).$$
Then $\sharp \mathcal{D}(f,\{H_i\}_{i=1}^q,\rho,2)\le 2$.
\end{theorem}
Remark: Suppose that $q=n+3$ and $\{H_i\}_{i=1}^{n+3}$ is in general position. Then the assumption of the above theorem is fulfilled with $\rho<\dfrac{1}{15n+9}$. Then this result is an extension of the finiteness theorems for differentiably non-degenerate meromorphic mappings into $\P^n(\C)$ sharing $n+3$ hyperplanes in general position of Quang \cite[Theorems 1.1,1.2,1.3]{Q14}.

\section{Basic notions and auxiliary results from the distribution theory}
 In this section, we recall some notations from the distribution value theory of meromorphic mappings on a ball $\B^m(\C)$ in $\C^m$ due to \cite{Q19,Q20}.

\noindent
{\bf 2.1. Counting function.}\ We set $\|z\| = \big(|z_1|^2 + \dots + |z_m|^2\big)^{1/2}$ for
$z = (z_1,\dots,z_m) \in \mathbb C^m$ and define
\begin{align*}
\B^m(R) &:= \{ z \in \mathbb C^m : \|z\| < R\}\ \   (0<R\le \infty),\\
S(R) &:= \{ z \in \mathbb C^m : \|z\| =R\}\ (0<R<\infty).
\end{align*}
Define 
$$v_{m-1}(z) := \big(dd^c \|z\|^2\big)^{m-1}\quad \quad \text{and}$$
$$\sigma_m(z):= d^c \text{log}\|z\|^2 \land \big(dd^c \text{log}\|z\|^2\big)^{m-1}
 \text{on} \quad \mathbb C^m \setminus \{0\}.$$

 For a divisor $\nu$ on  a ball $\B^m(R_0)$ of $\mathbb C^m$, and for a positive integer $p$ or $p= \infty$, we define the truncated counting function of $\nu$ by
$$n(t,\nu) =
\begin{cases}
\int\limits_{|\nu|\,\cap \B(t)}
\nu(z) v_{m-1} & \text  { if } m \geq 2,\\
\sum\limits_{|z|\leq t} \nu (z) & \text { if }  m=1. 
\end{cases}
$$
and define $n^{[p]}(t):=n(t,\nu^{[p]}),$ where $\nu^{[p]}=\min\{p,\nu\}.$

Define
$$ N(r,r_0,\nu)=\int\limits_{r_0}^r \dfrac {n(t)}{t^{2m-1}}dt \quad (0<r_0<r<R).$$

Similarly, define $N(r,r_0,\nu^{[p]})$ and denote it by $N^{[p]}(r,r_0,\nu)$.

Let $\varphi : \mathbb B^m(R_0) \longrightarrow \overline\C $ be a meromorphic function. Denote by $\nu_\varphi$ (res. $\nu^{0}_{\varphi}$) the divisor (resp. the zero divisor) of $\varphi$. Define
$$N_{\varphi}(r,r_0)=N(r,r_0,\nu^{0}_{\varphi}), \ N_{\varphi}^{[p]}(r,r_0)=N(r,r_0,(\nu^{0}_{\varphi})^{[p]}).$$
For brevity, we will omit the character $^{[p]}$ if $p=\infty$.

\vskip0.2cm
\noindent
{\bf 2.2. Characteristic function.}\ Throughout this paper, we fix a homogeneous coordinates system $(x_0:\cdots :x_n)$ on $\P^n(\C)$. Let $f : \mathbb B^m(R_0) \longrightarrow \mathbb P^n(\mathbb C)$ be a meromorphic mapping with a reduced representation $f = (f_0, \ldots,f_n)$, which means that each $f_i$ is a holomorphic function on $\B^m(R_0)$ and  $f(z) = \big(f_0(z) : \dots : f_n(z)\big)$ outside the indeterminancy locus $I(f)$ of $f$. Set $\|f \| = \big(|f_0|^2 + \dots + |f_n|^2\big)^{1/2}$.

The characteristic function of $f$ is defined by 
$$ T_f(r,r_0)=\int_{r_0}^r\dfrac{dt}{t^{2m-1}}\int\limits_{B(t)}f^*\Omega\wedge v^{m-1}, \ (0<r_0<r<R_0). $$
By Jensen's formula, we have
\begin{align*}
T_f(r,r_0)= \int\limits_{S(r)} \log\|f\| \sigma_m -
\int\limits_{S(r_0)}\log\|f\|\sigma_m +O(1), \text{ (as $r\rightarrow R_0$)}.
\end{align*}

If $R_0=+\infty$, we always choose $r_0=1$ and write $N_\varphi (r),N^{[p]}_\varphi (r),T_f(r)$ for $N_\varphi (r,1),$ $N^{[p]}_\varphi (r,1),T_f(r,1)$ as usual.
 
\vskip0.2cm
\noindent
{\bf 2.3. Auxiliary results.}\ Repeating the argument in \cite[Proposition 4.5]{Fu85}, we have the following.

\begin{proposition}\label{2.1}
Let $F_0,\ldots ,F_{l-1}$ be meromorphic functions on the ball $\B^{m}(R_0)$ in $\mathbb C^m$ such that $\{F_0,\ldots ,F_{l-1}\}$ are  linearly independent over $\mathbb C.$ Then  there exists an admissible set  
$$\{\alpha_i=(\alpha_{i1},\ldots,\alpha_{im})\}_{i=0}^{l-1} \subset \mathbb N^m,$$
which is chosen uniquely in an explicit way, with $|\alpha_i|=\sum_{j=1}^{m}|\alpha_{ij}|\le i \ (0\le i \le l-1)$ such that:

(i)\  $W_{\alpha_0,\ldots ,\alpha_{l-1}}(F_0,\ldots ,F_{l-1})\overset{Def}{:=}\det{({\mathcal D}^{\alpha_i}\ F_j)_{0\le i,j\le l-1}}\not\equiv 0.$ 

(ii) $W_{\alpha_0,\ldots ,\alpha_{l-1}}(hF_0,\ldots ,hF_{l-1})=h^{l+1}W_{\alpha_0,\ldots ,\alpha_{l-1}}(F_0,\ldots ,F_{l-1})$ for any nonzero meromorphic function $h$ on $\B^m(R_0).$
\end{proposition}
The function $W_{\alpha_0,\ldots ,\alpha_{l-1}}(F_0,\ldots ,F_{l-1})$ is called the general Wronskian of the mapping $F=(F_0,\ldots ,F_{l-1})$.

\begin{definition}[{Cartan's auxialiary function \cite[Definition 3.1]{Fu98}}]\label{2.2}
For meromorphic functions $F,G,H$ on $\B^m(R_0)$ and $\alpha =(\alpha_1,\ldots ,\alpha_m)\in \Z_+^m$, we define the Cartan's auxiliary function as follows:
$$
\Phi^\alpha(F,G,H):=F\cdot G\cdot H\cdot\left | 
\begin {array}{cccc}
1&1&1\\
\frac {1}{F}&\frac {1}{G} &\frac {1}{H}\\
\mathcal {D}^{\alpha}(\frac {1}{F}) &\mathcal {D}^{\alpha}(\frac {1}{G}) &\mathcal {D}^{\alpha}(\frac {1}{H})
\end {array}
\right|.
$$
\end{definition}

\begin{lemma}[{see \cite[Proposition 3.4]{Fu98}}]\label{2.3} If $\Phi^\alpha(F,G,H)=0$ and $\Phi^\alpha(\frac {1}{F},\frac {1}{G},\frac {1}{H})=0$ for all $\alpha$ with $|\alpha|\le 1$, then one of the following assertions holds:

(i) \ $F=G, G=H$ or $H=F$,

(ii) \ $\frac {F}{G},\frac {G}{H}$ and $\frac {H}{F}$ are all constant.
\end{lemma}

\vskip0.2cm
\noindent
{\bf 2.3. Functions of small integration and bounded integration.}\ 
Let $f^1,f^2,\ldots,f^k$ be $k$ meromorphic mappings from the complete K\"{a}hler manifold $\B^m(R_0)$ into $\P^n(\C)$, which satisfies the condition $(C_\rho)$ for a non-negative number $\rho$. 
For each $1\le u\le k$, we fix a reduced representation $f^u=(f^u_0:\cdots :f^u_n)$ of $f^u$ and set $\|f^u\|=(|f^u|^2_0+\cdots+|f^u|^2_n)^{1/2}$.

We denote by $\mathcal C(\B^m(R_0))$ the set of all non-negative functions $g: \B^m(R_0)\to [0,+\infty]$ which are continuous outside an analytic set of codimension two (corresponding to the topology of the compactification $[0,+\infty]$) and only attain $+\infty$ in an analytic thin set.

\begin{definition}[{see \cite[Definition 2.2]{Q19} and \cite[Definition 3.1]{Q20}}]\label{3.1}
A function $g$ in $\mathcal C(\B^m(R_0))$ is said to be of small integration with respective to $f^1,\ldots,f^k$ at level $l_0$ if there exist an element $\alpha=(\alpha_1,\ldots,\alpha_m)\in\N^m$ with $|\alpha|\le l_0$,  a positive number $K$, such that for every $0\le tl_0<p<1$,
$$\int_{S(r)}|z^\alpha g|^t\sigma_m \le K\left(\frac{R^{2m-1}}{R-r}\sum_{u=1}^kT_{f^u}(r,r_0)\right)^p$$
for all $r$ with $0<r_0<r<R<R_0$, where $z^\alpha=z_1^{\alpha_1}\cdots z_m^{\alpha_m}$. 
\end{definition}

We denote by $S(l_0;f^1,\ldots,f^k)$ the set of all functions in $\mathcal C(\B^m(R_0))$ which are of small integration with respective to $f^1,\ldots,f^k$ at level $l_0$. We see that, if $g$ belongs to $S(l_0;f^1,\ldots,f^k)$ then $g$ is also belongs to $S(l;f^1,\ldots,f^k)$ for every $l>l_0$. Moreover, if $g$ is a constant function then $g\in S(0;f^1,\ldots,f^k)$.

\begin{proposition}[{see \cite[Proposition 2.3]{Q19} and \cite[Proposition 3.2]{Q20}}]\label{3.2}
If $g_i\in S(l_i;f^1,\ldots,f^l)$ $(1\le i\le s)$ then $\prod_{i=1}^sg_i\in S(\sum_{i=1}^sl_i;f^1,\ldots,f^l)$.
\end{proposition}

\begin{definition}[{see \cite[Definition 3.3]{Q20}}]\label{3.3} 
A meromorphic function $h$ on $\B^m(R_0)$ is said to be of bounded integration with bi-degree $(p,l_0)$ for the family $\{f^1,\ldots,f^k\}$ if there exists $g\in S(l_0;f^1,\ldots,f^k)$ satisfying
$$|h|\le \|f^1\|^p\cdots \|f^u\|^p\cdot g,$$
outside a proper analytic subset of $\B^m(R_0)$.
\end{definition}
Denote by $B(p,l_0;f^1,\ldots,f^k)$ the set of all meromorphic functions on $\B^m(R_0)$ which are of bounded integration of bi-degree $(p,l_0)$ for $\{f^1,\ldots,f^k\}$. We have the following:
\begin{itemize}
\item For a meromorphic function $h$, $|h|\in S(l_0;f^1,\ldots,f^k)$ iff $h\in B(0,l_0;f^1,\ldots,f^k)$.
\item $B(p,l_0;f^1,\ldots,f^k)\subset B(p,l;f^1,\ldots,f^k)$ for every $0\le l_0<l$.
\item If $h_i\in B(p_i,l_i;f^1,\ldots,f^k)\ (1\le i\le s)$ then 
$$h_1\cdots h_m\in B(\sum_{i=1}^sp_i,\sum_{i=1}^sl_i;f^1,\ldots,f^k).$$
\end{itemize}

The following proposition is proved by Fujimoto \cite{Fu83} and reproved by Ru-Sogome \cite{RS}.
\begin{proposition}[{see \cite[Proposition 6.1]{Fu83}, also \cite[Proposition 3.3]{RS}}]\label{3.4}
 Let $L_1,\ldots ,L_l$ be linear forms of $l$ variables and assume that they are linearly independent. Let $F$ be a meromorphic mapping from the ball $\B^m(R_0)\subset\C^m$ into $\P^{l-1}(\C)$ with a reduced representation $F=(F_0,\ldots ,F_{l-1})$ and let $(\alpha_1,\ldots ,\alpha_l)$ be an admissible set of $F$. Set $l_0=|\alpha_1|+\cdots +|\alpha_l|$ and take $t,p$ with $0< tl_0< p<1$. Then, for $0 < r_0 < R_0,$ there exists a positive constant $K$ such that for $r_0 < r < R < R_0$,
$$\int_{S(r)}\biggl |z^{\alpha_1+\cdots +\alpha_l}\dfrac{W_{\alpha_1,\ldots ,\alpha_l}(F_0,\ldots ,F_{l-1})}{L_0(F)\ldots L_{l-1}(F)}\biggl |^t\sigma_m\le K\biggl (\dfrac{R^{2m-1}}{R-r}T_F(R,r_0)\biggl )^p.$$
\end{proposition}
This proposition implies that the function $\left |\dfrac{W_{\alpha_1,\ldots ,\alpha_l}(F_0,\ldots ,F_{l-1})}{L_0(F)\ldots L_{l-1}(F)}\right |$ belongs to $S(l_0;F)$.

\begin{lemma}[{see also \cite[Lemma 3.3 and Lemma 3.4]{No05}}]\label{5.1}
Let $H_1,...,H_q$ be $q$ hyperplanes in $\mathbb P^n(\mathbb C)$ in $N$-subgeneral position, where $q>2N-n+1.$ Then, there are positive rational constants $\omega_i\ (1\le i\le q)$ satisfying the following:

i) $0<\omega_i \le 1,\  \forall i\in\{1,...,q\}$,

ii) Setting $\tilde \omega =\max_{j\in Q}\omega_j$, one gets
$$\sum_{j=1}^{q}\omega_j=\tilde \omega (q-2N+n-1)+n+1.$$

iii) $\dfrac{n+1}{2N-n+1}\le \tilde\omega\le\dfrac{n}{N}.$

iv) Let $E_i\ge 1\ (1\le i \le q)$ be arbitrarily given numbers. For $R\subset \{1,...,q\}$ with $\sharp R = N+1$,  there is a subset $R^o\subset R$ such that $\sharp R^o=\rank \{H_i\}_{i\in R^o}=n+1$ and 
$$\prod_{i\in R}E_i^{\omega_i}\le\prod_{i\in R^o}E_i.$$
\end{lemma}

\section{Proof of Theorem \ref{thm1.1}}

In this section we will prove Theorem \ref{thm1.1}. We need the following lemmas.
\begin{lemma}\label{lem3.1}
Let $f$ be a differentiably non-degenerate meromorphic mapping of a ball $\B^m(R_0)$ in $\C^m$ into $\P^n(\C)\ (m\ge n)$ with a reduced representation $(f_0:\cdots:f_n)$. Let $H_0,\ldots,H_n$ be $n+1$ hyperplanes of $\P^n(\C)$ in general possition. Let $\alpha =(\alpha_0,\ldots,\alpha_n)\in (\N^m)^{n+1}$ with $|\alpha_0|=0,|\alpha_i|=1\ (1\le i\le n)$ such that $W:=\det (\mathcal D^{\alpha_i}f_j; 0\le i,j\le n)\not\equiv 0$. Then we have
$$ \sum_{i=0}^n\nu_{(f,H_i)}-\nu_W\le\nu^{[1]}_{\prod_{i=0}^n(f,H_i)}.$$
\end{lemma}
\begin{proof}
Since $W=C\det (\mathcal D^{\alpha_i}(f,H_j))$ with a nonzero constant $C$, without loss of generality we may suppose that $H_i=\{\omega_i=0\}\ (0\le i\le n).$ Then we have $(f,H_i)=f_i$. Also, we may assume that 
$$\alpha_1=(1,0,0,\ldots,0),\alpha_2=(0,1,0,\ldots,0),\ldots,\alpha_n=(0,0,\ldots,0,\overset{n-th}{1},0\ldots,0).$$
Let $b$ be a regular point of the analytic set $S=\{f_0\cdots f_n=0\}$ and $b$ is not in the indeterminacy locus $I(f)$ of $f$. Then there is a local affine coordinates $(U,x)$ around $b$, where $U$ is a neighborhood of $b$ in $\B^m(R_0)$, $x=(x_1,\ldots,x_m),x(b)=(0,\ldots,0)$ such that $S\cap U=\{x_1=0\}\cap U$. 

Since $b\not\in I(f)$, we may suppose that $S\cap U=\{f_i=0\}\cap U\ (0\le i\le l)$ and $f_j\ (l+1\le j\le n)$ does not vanishes on $U$. Therefore, we have $f_i=x_1^{t_i}g_j\ (0\le i\le l)$ with some holomorphic function $g_j$. We easily see that
$$ \mathcal D^{\alpha_i}(f_j/f_n)=\frac{\partial (f_j/f_n)}{\partial z_i}=\sum_{s=1}^m\frac{\partial x_s}{\partial z_i}\cdot \frac{\partial}{\partial x_s}\left(\frac{f_j}{f_n}\right)\ (0\le j\le n-1) $$  
and 
$$ \nu_{\frac{\partial}{\partial x_s}\left(\frac{f_j}{f_n}\right)}(b)\ge \begin{cases}
t_j-1&\text{if }s=1\\
t_j&\text{if }s>1,
\end{cases}\ \forall 1\le j\le l.$$
On the other hand, we have
\begin{align*}
W=&\det (D^{\alpha_i}f_j;0\le i,j\le n)=
\left |\begin{array}{cccc}
f_0&f_1&\ldots&f_n\\ 
\frac{\partial f_0}{\partial z_1}&\frac{\partial f_1}{\partial z_1}&\ldots&\frac{\partial f_n}{\partial z_1}\\
\vdots&\vdots&\ldots&\vdots\\ 
\frac{\partial f_0}{\partial z_n}&\frac{\partial f_1}{\partial z_n}&\ldots&\frac{\partial f_n}{\partial z_n}
\end{array}\right |\\
=&f_n^{n+1}\left |\begin{array}{cccc}
\frac{\partial (f_0/f_n)}{\partial z_1}&\frac{\partial (f_1/f_n)}{\partial z_1}&\ldots&\frac{\partial (f_{n-1}/f_n)}{\partial z_1}\\
\vdots&\vdots&\ldots&\vdots\\ 
\frac{\partial (f_0/f_n)}{\partial z_n}&\frac{\partial (f_1/f_n)}{\partial z_n}&\ldots&\frac{\partial(f_{n-1}/f_n)}{\partial z_n}
\end{array}\right|.
\end{align*}
This implies that
\begin{align*}
\nu_W(b)&\ge\min\{\nu_{\det\left(\frac{\partial}{\partial x_{i_s}}\bigl(\frac{f_j}{f_n}\bigl);0\le j,s\le n-1\right)}(b); 1\le i_0<\cdots <i_{n-1}\le m\}\\
&\ge\min\sum_{j=0}^{n-1}\nu_{\frac{\partial}{\partial x_{i_j}}\bigl(\frac{f_j}{f_n}\bigl)}(b)\\
&\ge t_1+\cdots+t_l-1=\sum_{i=0}^{n}\nu_{(f,H_i)}(b)-\nu^{[1]}_{\prod_{i=0}^n(f,H_i)}(b). 
\end{align*}
Therefore, we have 
$$ \sum_{i=0}^n\nu_{(f,H_i)}(b)-\nu^{[1]}_{\prod_{i=0}^n(f,H_i)}(b)\ge \nu_W(b).$$
The lemma is proved.
\end{proof}

\begin{lemma}\label{lem3.2} 
	Let $f^1,f^2,\ldots,f^k$ be $k$ differentiably nondegenerate meromorphic mappings from the complete K\"{a}hler manifold whose universal covering is biholomorphic to $\B^m(R_0)$ into $\P^n(\C)$, which satisfy the condition $(C_\rho)$. Let $H_1,\ldots,H_q$ be $q$ hyperplanes of $\P^n(\C)$ in $N-$subgeneral position, where $q$ is a positive integer. Assume that there exists a non zero holomorphic function $h\in B(p,l_0;f^1,\ldots,f^k)$ such that
	$$\nu_h\ge\lambda\sum_{u=1}^k\nu^{[1]}_{(f^u,D)},$$
	where $D$ is the hypersurface $H_1+\cdots+H_q$, $p,l_0$ are non-negative integers, $\lambda$ is a positive number. Then we have
\begin{align}\label{eq3.3}
q\le 2N-n+1+\frac{p(2N-n+1)}{\lambda(n+1)}+\rho\left (kn+\dfrac{l_0}{\lambda}\right ).
\end{align}
Moreover, if we assume further that $\nu_h\ge\lambda\sum_{u=1}^k\sum_{i=1}^q\nu^{[1]}_{(f^u,H_i)}$ then we have
\begin{align}\label{eq3.4}
q\le 2N-n+1+\frac{p(2N-n+1)}{\lambda (n+1)}+\rho\left (kn+\dfrac{l_0n}{\lambda N}\right ).
\end{align}
\end{lemma}

\noindent
\textit{Proof}.\
Since each $f^u$ is differentiably nondegenerate, $df^u$ has the rank $n$ at some points outside the indeterminacy locus of $f^u$. Hence, there exist indices $\alpha^u =(\alpha^u_0,\ldots,\alpha^u_n)\in (\N^m)^{n+1}$ with $|\alpha^u_0|=0,|\alpha^u_i|=1\ (1\le i\le n)$ such that 
\begin{align}\label{eq3.5}
\begin{split}
W^u&:=\det (\mathcal D^{\alpha^u_i}f^u_j; 0\le i,j\le n)\\
&=(f^u_n)^{n+1}\det (\mathcal D^{\alpha^u_i}(f^u_j/f^u_n); 1\le i\le n, 0\le j\le n-1)\not\equiv 0.
\end{split}
\end{align}
For each $R^o=\{r^o_1,...,r^o_{n+1}\}\subset\{1,...,q\}$ with $\rank \{H_i\}_{i\in R^o}=\sharp R^o=n+1$, we set 
$$W^u_{R^o}\equiv  \det (\mathcal D^{\alpha^u_i}(f^u,H_{r^0_j}); 0\le i\le n,1\le j\le n+1).$$
Denote by $\tilde\omega, \omega_i\ (1\le i\le q)$ the Nochka's weights of the family $\{H_i\}_{i=1}^q$. We need the following two claims.

\begin{claim}\label{cl3.6}
$\sum_{i=1}^q\omega_i\nu_{(f^u,H_i)}(z)-\nu_{W^u}(z)\le \nu^{[1]}_{(f^u,D)}$.
\end{claim}
Indeed, assume that  $z$ is a zero of some $(f^u,H_i)(z)$ and $z$ is outside the indeterminancy locus $I(f^u)$ of $f^u$. Since $\{H_i\}_{i=1}^q$ is in $N$-subgeneral position, it implies that $z$ is not zero of more than $N$ functions $(f^u,H_i)$. Without loss of generality, we may assume that $z$ is not zero of $(f^u,H_i)$ for each $i>N$. Put $R=\{1,...,N+1\}.$ Choose $R^1\subset R$ such that 
$$\sharp R^1=\rank\{H_i\}_{i\in R^1}=n+1$$
and $R^1$ satisfies Lemma \ref{5.1} iv) with respect to numbers $\bigl \{e^{\nu_{H_i(f^u)}(z)} \bigl \}_{i=1}^q.$
Then we have
\begin{align*}
\sum_{i\in R}\omega_i\nu_{(f^u,H_i)}(z)\le \sum_{i\in R^1}\nu_{(f^u,H_i)}(z).
\end{align*}
By Lemma \ref{lem3.1}, this implies that
\begin{align*}
\nu_{W^u}(z)= \nu_{W^u_{R^1}}(z)& \ge \sum_{i\in R^1}\nu_{(f^u,H_i)}(z)-\nu^{[1]}_{\prod_{s\in R^1}(f^u,H_s)}(z).
\end{align*}
Hence, we have 
\begin{align*}
\sum_{i=1}^q\omega_i\nu_{(f^u,H_i)}(z)-\nu_{W^u}(z)\le \min\{1,\nu_{\prod_{s=1}^q(f^u,H_i)}(z)\}=\nu^{[1]}_{(f^u,D)}(z).
\end{align*}
The claim is proved.

By Claim \ref{cl3.6}, we see that 
$$\nu^{[1]}_{(f^u,D)}\ge \sum_{i=1}^q\omega_i\nu_{(f^u,H_i)}-\nu_{W^u}.$$
Then we have 
\begin{align}\label{eq3.7}
\nu_h\ge\lambda\sum_{u=1}^k\nu^{[1]}_{(f^u,D)}\ge\lambda\sum_{u=1}^k\left (\sum_{i=1}^q\omega_i\nu_{(f^u,H_i)}-\nu_{W^u}\right).
\end{align}

On the other hand we also have the following claim.
\begin{claim}\label{cl3.8}
$\sum_{i=1}^q\omega_i\nu_{(f^u,H_i)}(z)-\nu_{W^u}(z)\le \sum_{i=1}^q\omega_i\min\{1,\nu_{(f^u,H_i)}(z)\}$.
\end{claim}
Indeed, assume that  $z$ is a zero of some $(f^u,H_i)$'s and $z$ is outside $I(f^u)$. Then $z$ is not zero of more than $N$ functions $(f^u,H_i)$. Without loss of generality, we may assume that $z$ is not zero of $(f^u,H_i)$ for each $i>N$. Put $R=\{1,...,N+1\}.$ Choose $R^2\subset R$ such that 
$$\sharp R^2=\rank\{H_i\}_{i\in R^2}=n+1$$
and $R^2$ satisfies Lemma \ref{5.1} iv) with respect to numbers $\bigl \{e^{\max\{\nu_{(f^u,H_i)}(z)-1,0\}} \bigl \}_{i=1}^q.$
 Then we have
\begin{align*}
\sum_{i\in R}\omega_i \max\{\nu_{(f^u,H_i)}(z)-1,0\} \le \sum_{i\in R^2}\max\{\nu_{(f^u,H_i)}(z)-1,0\}.
\end{align*}
This implies that
\begin{align*}
\nu_{W^u}(z)= \nu_{W_{R^2}}(z)& \ge \sum_{i\in R^2}\max\{\nu_{(f^u,H_i)}(z)-1,0\}\ge\sum_{i\in R}\omega_i \max\{\nu_{(f^u,H_i)}(z)-1,0\}. 
\end{align*}
Hence, we have 
\begin{align*}
\sum_{i=1}^q\omega_i\nu_{(f^u,H_i)}(z)-\nu_{W^u}(z) &=\sum_{i\in R}\omega_i\nu_{(f^u,H_i)}(z)-\nu_{W^u}(z)\\ 
&= \sum_{i\in R}\omega_i\min\{\nu_{(f^u,H_i)}(z),1\}\\
&+\sum_{i\in R}\omega_i\max\{\nu_{(f^u,H_i)}(z)-1,0\}-\nu_{W^u}(z)\\
&\le \sum_{i\in R}\omega_i\min\{\nu_{(f^u,H_i)}(z),1\}= \sum_{j=1}^q\omega_j\nu_{\varphi_j}(z).
\end{align*}
The claim is proved.

Hence, if we assume moreover that
$$\nu_h\ge\lambda\sum_{u=1}^k\sum_{i=1}^q\nu^{[1]}_{(f^u,H_i)}$$ 
then by Claim \ref{cl3.8} we have
\begin{align}\label{3.9}
\nu_h\ge\lambda \frac{n}{N}\sum_{u=1}^k\sum_{i=1}^q\omega_i\nu^{[1]}_{(f^u,H_i)}\ge\lambda\frac{n}{N}\sum_{u=1}^k\left (\sum_{i=1}^q\omega_i\nu_{(f^u,H_i)}-\nu_{W^u}\right).
\end{align}

From (\ref{eq3.7}) and (\ref{3.9}), in order to prove Lemma \ref{lem3.2} we only need to proof the following.
\begin{lemma}\label{lem3.10}
Let $f^1,f^2,\ldots,f^k$ and $H_1,\ldots,H_q$ be as in Theorem \ref{lem3.2}. Assume that there exists a non zero holomorphic function $h\in B(p,l_0;f^1,\ldots,f^k)$ such that
$$\nu_h\ge\lambda\sum_{u=1}^k\left (\sum_{i=1}^q\omega_i\nu_{(f^u,H_i)}-\nu_{W^u}\right ).$$
Then we have
	$$q\le 2N-n+1+\frac{p(2N-n+1)}{\lambda(n+1)}+\rho\left (kn+\dfrac{l_0}{\lambda}\right ).$$
\end{lemma}
\begin{proof} 
If $R_0=+\infty$, by usual argument in Nevanlinna theory (see \cite[ineq. (3.11)-(3.12)]{No05}), we have
\begin{align*}
(q-2N+n-1)\sum_{u=1}^kT_{f^u}(r)&\le\sum_{u=1}^k\dfrac{1}{\tilde\omega}\left (\sum_{i=1}^q\omega_iN_{(f^u,H_i)}(r)-N_{W^u}(r)\right)+o(\sum_{u=1}^kT_{f^u}(r))\\
&\le \frac{2N-n+1}{\lambda(n+1)}N_h(r)+o(\sum_{u=1}^kT_{f^u}(r))\\
&\le\frac{p(2N-n+1)}{\lambda(n+1)}\sum_{u=1}^kT_{f^u}(r)+o(\sum_{u=1}^kT_{f^u}(r)),
\end{align*}
for all $r\in [1;+\infty)$ outside a Lebesgue set of finite measure. 
Letting $r\rightarrow +\infty$, we obtain
$$ q\le 2N-n+1+\frac{p(2N-n+1)}{\lambda(n+1)}.$$

Now, we consider the case where $R_0<+\infty$. Without loss of generality we assume that $R_0=1$. Suppose contrarily that $q>2N-n+1+\frac{p(2N-n+1)}{\lambda(n+1)}+\rho\left (kn+\dfrac{l_0}{\lambda}\right )$. Then, there is a positive constant $\epsilon$ such that
	$$q> 2N-n+1+\frac{p(2N-n+1)}{\lambda(n+1)}+\rho\left (kn+\dfrac{l_0+\epsilon}{\lambda}\right ).$$ 
	Put $l_0'=l_0+\epsilon >0$.

	Put $\zeta_u(z):=\left |z^{\alpha_0^u+\cdots +\alpha_n^u}\dfrac{W^{\alpha^u}(f^u)}{\prod_{i=1}^{q}|(f,H_i)|^{\omega_i}}\right |\ (1\le u\le k)$. Since $h\in B(p,l_0;f^1,\ldots,f^k)$, there exists a function $g\in S(l_0;f^1,\ldots,f^k)$ and $\beta =(\beta_1,\ldots,\beta_m)\in\mathbb Z^{m}_+$ with $|\beta|\le l_0$ such that 
	\begin{align}\label{eq3.11}
	\int_{S(r)}\left |z^\beta g\right|^{t'}\sigma_m=O\left (\frac{R^{2m-1}}{R-r}\sum_{u=1}^kT_{f^u}(r,r_0)\right )^l,
	\end{align}
for every $0\le l_0t'<l<1$ and 
\begin{align}\label{eq3.12}
|h|\le \left (\prod_{u=1}^k\|f^u\|\right )^p |g|.
\end{align}

	Put $t=\frac{\rho}{\tilde\omega(q-2N+n-1)-\frac{p}{\lambda}}>0$ (since $q-2N+n-1-\frac{p}{\lambda\tilde\omega}>q-2N+n-1-\frac{p(2N-n+1)}{\lambda(n+1)}$) and $ \phi:=|\zeta_1|\cdots |\zeta_k|\cdot|z^\beta h|^{1/\lambda}.$
	Then $a=t\log\phi$ is a plurisubharmonic function on $\B^m(1)$ and
\begin{align*}
\left (kn+\frac{l_0'}{\lambda}\right)t&\le \left (kn+\frac{l_0'}{\lambda}\right)\frac{\rho}{\tilde\omega(q-2N+n-1)-\frac{p}{\lambda}}\\
& \le \left (kn+\frac{l_0'}{\lambda}\right)\frac{\rho (2N-n+1)}{(q-2N+n-1)(n+1)-\frac{p (2N-n+1)}{\lambda}}<1.
\end{align*}
	Therefore, we may choose a positive number $p'$ such that $0\le (kn+\frac{l_0'}{\lambda})t<p'<1.$
	
	Since $f^u$ satisfies the condition $(C_\rho)$, then there exists a continuous plurisubharmonic function $\varphi_u$ on $\B^m(1)$ such that
	$$ e^{\varphi_u}dV \le \|f^u\|^{\rho}v_m.$$
	We see that $\varphi=\varphi_1+\cdots+\varphi_k+a$ is a plurisubharmonic function on $\B^m(1)$. We have
	\begin{align*}
	e^\varphi dV&=e^{\varphi_1+\cdots+\varphi_k+t\log\phi}dV\le e^{t\log\phi}\prod_{u=1}^k\|f^u\|^{\rho}v_m=|\phi|^{t}\prod_{u=1}^k\|f^u\|^{\rho}v_m\\
	&\le |z^\beta g|^{t/\lambda}\prod_{u=1}^k(|\zeta_u|^t\cdot\|f^u\|^{\rho+pt/\lambda})v_m=|z^\beta g|^{t/\lambda}\prod_{u=1}^k(|\zeta_u|^t\cdot\|f^u\|^{\tilde\omega(q-2N+n-1)t})v_m.
	\end{align*}
	Setting $x=\dfrac{l_0'/\lambda}{kn+l_0'/\lambda}$, $y=\dfrac{n}{kn+l_0'/\lambda}$, then we have $x+ky=1$. Therefore, by integrating both sides of the above inequality over $\B^m(1)$ and applying H\"{o}lder inequality,  we have
	\begin{align}\label{eq3.13}
	\begin{split}
	\int_{\B^m(1)}e^\varphi dV&\le \int_{\B^m(1)}\prod_{u=1}^k(|\zeta_u|^t\cdot\|f^u\|^{\tilde\omega(q-2N+n-1)t})|z^\beta g|^{t/\lambda}v_m\\
	&\le \left (\int_{\B^m(1)}|z^\beta g|^{t/(\lambda x)}v_m\right)^{x}\\
	&\ \ \times\prod_{u=1}^k\left (\int_{\B^m(1)}(|\zeta_u|^{t/y}\cdot\|f^u\|^{\tilde\omega(q-2N+n-1)t/y})v_m\right)^{y}\\
	&\le \left (2m\int_0^1r^{2m-1}\left (\int_{S(r)}|z^\beta g|^{t/(\lambda x)}\sigma_m\right )dr\right)^{x}\\
	&\ \ \times\prod_{u=1}^k\left (2m\int_0^1r^{2m-1}\left (\int_{S(r)}\bigl (|\zeta_u|\cdot\|f^u\|^{(\sum_{i=1}^q\omega_i-n-1)}\bigl )^{t/y}\sigma_m\right )dr\right)^{y}.
	\end{split}
	\end{align}
	
	(a) We now deal with the case where
	$$ \lim\limits_{r\rightarrow 1}\sup\dfrac{\sum_{u=1}^kT_{f^u}(r,r_0)}{\log 1/(1-r)}<\infty .$$
	We see that $\dfrac{l_0t}{\lambda x}\le \dfrac{l_0't}{\lambda x}=\bigl (kn+\dfrac{l_0'}{\lambda}\bigl )t<p'$ and $n\dfrac{t}{y}=\bigl (kn+\dfrac{l_0'}{\lambda}\bigl )t<p'$.  By lemma on logarithmic derivative there exists a positive constant $K$ such that, for every $0<r_0<r<r'< 1,$ we have
	\begin{align*}
	&\int_{S(r)}\bigl (|\zeta_u|\cdot\|f^u\|^{(\sum_{i=1}^q\omega_i-n-1)}\bigl )^{t/y}\sigma_m\le K\left (\dfrac{r'^{2m-1}}{r'-r}T_{f^u}(r',r_0)\right )^{p'}\ (1\le u\le k)\\
	\text{and }&\int_{S(r)}|z^\beta g|^{t/(\lambda x)}\sigma_m\le K\left (\dfrac{r'^{2m-1}}{r'-r}\sum_{u=1}^kT_{f^u}(r',r_0)\right )^{p'}.
	\end{align*}
		
	Choosing $r'=r+\dfrac{1-r}{e\max_{1\le u\le k}T_{f^u}(r,r_0)}$, we have
	$ T_{f^u}(r',r_0)\le 2T_{f^u}(r,r_0)$,
	for all $r$ outside a subset $E$ of $(0,1]$ with $\int_E\frac{1}{1-r}dr<+\infty$.
	Hence, the above inequality implies that
	\begin{align*}
	&\int_{S(r)}\bigl (|w_u|\cdot\|f^u\|^{(\sum_{i=1}^q\omega_i-n-1)}\bigl )^{t/y}\sigma_m\le \dfrac{K'}{(1-r)^{p'}}\left (\log\frac{1}{1-r}\right )^{2p'}\ (1\le u\le k)\\
	\text{and }&\int_{S(r)}|z^\beta g|^{t/(\lambda x)}\sigma_m\le \dfrac{K'}{(1-r)^{p'}}\left (\log\frac{1}{1-r}\right )^{2p'}
	\end{align*}
	for all $r$ outside $E$, and for some positive constant $K'$. Then the inequality (\ref{eq3.13}) yields that
	\begin{align*}
	\int_{\B^m(1)}e^udV&\le 2m\int_0^1r^{2m-1}\frac{K'}{1-r}\left (\log\frac{1}{1-r}\right)^{2p'}dr< +\infty.
	\end{align*}
	This contradicts the results of S.T. Yau \cite{Y76} and L. Karp \cite{K82}. 
	
	(b) We now deal with the remaining case where 
	$$ \lim\limits_{r\rightarrow 1} \sup\dfrac{\sum_{u=1}^kT_{f^u}(r,r_0)}{\log 1/(1-r)}= \infty.$$
	As above, we have
	$$\int_{S(r)}|z^\beta g|^{t/(\lambda x)}\sigma_m\le K\left (\dfrac{1}{1-r}\sum_{u=1}^kT_{f^u}(r,r_0)\right )^{p'}$$
	for every $r_0<r<1.$
	By the concativity of the logarithmic function, we have
	$$ \int_{S(r)}\log|z^\beta|^{t/(\lambda x)}\sigma_m+\int_{S(r)}\log|g|^{t/(\lambda x)}\sigma_m\le K''\left (\log^+\frac{1}{1-r}+\log^+\sum_{u=1}^kT_{f^u}(r,r_0)\right). $$
	This implies that
	$$\int_{S(r)}\log|g|\sigma_m= O\left (\log^+\frac{1}{1-r}+\log^+\sum_{u=1}^kT_{f^u}(r,r_0)\right) $$
	By \eqref{eq3.12}, we have
	\begin{align*}
	\sum_{u=1}^kpT_{f^u}(r,r_0)+\int_{S(r)}\log|g|\sigma_m&\ge N_h(r,r_0)+S(r)\ge\lambda\sum_{u=1}^kN_{(f,D)}^{[1]}(r,r_0)+S(r)\\ 
	&\ge \lambda\sum_{u=1}^k\dfrac{(q-2N+n-1)(n+1)}{2N-n+1}T_{f^u}(r,r_0)+S(r), 
	\end{align*}
	where $S(r)=O(\log^+\frac{1}{1-r}+\log^+\sum_{u=1}^kT_{f^u}(r_0,r))$ for every $r$ excluding a set $E$ with $\int_E\frac{dr}{1-r}<+\infty$.
	Letting $r\rightarrow 1$, we get 
$$\frac{p}{\lambda}>\frac{(q-2N+n-1)(n+1)}{2N-n+1},$$
i.e.,
$$q<2N-n+1+\frac{p(2N-n+1)}{\lambda (n+1)}.$$
This is a contradiction.

Hence, the supposition is false. The proposition is proved.
\end{proof}

\vskip0.2cm
Now consider $k$ mappings $f^1, \ldots, f^k \in \mathcal{D}(f,\{H_i\}_{i=1}^{q},1)$. We denote $\Gamma$ the set of all irreducible component of $\bigcup_{i=1}^q\{z:\ (f,H_i)(z)=0\}$. For each $\gamma\in \Gamma$, we define $V^u_{\gamma}$ to be the set of all $(c_0,\ldots,c_n)\in\C^{n+1}$ such that
$$\gamma\subset\{z\ :\ c_0f^u_0(z)+\cdots +c_nf^u_n(z)=0\}\ (1\le u\le k).$$
It easy to see that $V^u_\gamma$ is a proper vector subspace of $\C^{n+1}$. Then $\bigcup_{u=1}^k\bigcup_{\gamma\in\Gamma}V^u_\gamma$ is the union of finite proper vector spaces of $\C^{n+1}$, and then is nowhere density $\C^{n+1}$. We set 
\begin{align}\label{eq3.14}
\mathcal C:=\C^{n+1}\setminus \bigcup_{u=1}^k\bigcup_{\gamma\in\Gamma}V^u_\gamma,
\end{align}
then $\mathcal C$ is a density open subset of $\C^{n+1}$, and hence there exists $c=(c_0,\ldots,c_n)\in\mathcal C$. By changing the coordinates if necessary, without loss of generality, from here we always assume that $c=(1,0,\ldots,0)\in\mathcal C$. Then we have
$$ \dim \{z\ :\ f^u_0(z)=0\}\cap\{z\ :\ \prod_{i=1}^q(f,H_i)(z)=0\}\le m-2\ (1\le u\le k).$$

\begin{proof}[Proof of Theorem \ref{thm1.1}]
(a) Suppose that $f^1\wedge\ldots\wedge f^k\not\equiv 0$. Then there $k$ indices $0\le i_1<\cdots<i_k\le n$ such that
$$ P:= \mathrm{det}\left (\begin{array}{ccc}
f^1_{i_1}&\cdots&f^k_{i_1}\\ 
\vdots&\cdots&\vdots\\ 
f^1_{i_k}&\cdots&f^k_{i_k}
\end{array}\right )\not\equiv 0.$$
We have 
\begin{align*}
P=f^1_{0}\cdots f^k_{0}\cdot
\left |\begin{array}{cccc}
1&1&\cdots&1\\ 
\frac{f^1_{i_2}}{f^1_{i_1}}&\frac{f^2_{i_2}}{f^1_{i_1}}&\cdots&\frac{f^k_{i_2}}{f^1_{i_1}}\\ 
\vdots&\vdots&\cdots&\vdots\\ 
\frac{f^1_{i_k}}{f^1_{i_1}}&\frac{f^2_{i_k}}{f^1_{i_1}}&\cdots&\frac{f^k_{i_k}}{f^1_{i_1}}
\end{array}\right |=f^1_{0}\cdots f^k_{0}\cdot
\left |\begin{array}{cccc}
\frac{f^2_{i_2}}{f^1_{i_1}}-\frac{f^1_{i_2}}{f^1_{i_1}}&\cdots&\frac{f^k_{i_2}}{f^1_{i_1}}-\frac{f^1_{i_2}}{f^1_{i_1}}\\ 
\vdots&\cdots&\vdots\\ 
\frac{f^2_{i_k}}{f^1_{i_1}}-\frac{f^1_{i_k}}{f^1_{i_1}}&\cdots&\frac{f^k_{i_k}}{f^1_{i_1}}-\frac{f^1_{i_k}}{f^1_{i_1}}
\end{array}\right |.
\end{align*}
Hence, if a point $z\not\in\bigcup_{u=1}^k\{f^u_0=0\}$ is a zero of $(f,D)$ then it will be a zero of $P$ with multiplicity at least $k-1$. Therefore, we have
\begin{align*}
\nu_{P}\ge (k-1)\nu_{(f,D)}^{[1]}=\frac{k-1}{k}\sum_{u=1}^k\nu^{[1]}_{(f^u,D)}.
\end{align*}
It also is easy to see that $P\in B(1,0;f^1,\ldots,f^k)$. Then, by Proposition \ref{lem3.2} we have
$$q\le 2N-n+1+\frac{k(2N-n+1)}{(k-1)(n+1)}+kn\rho.$$
This is a contradiction. 

Then $f^1\wedge\cdots\wedge f^k\equiv 0$. The assertion (a)  is proved.

(b) Using the same notation and repeating the same argument as in the above part, we have 
\begin{align*}
\nu_{P}\ge (k-1)\nu_{(f,D)}^{[1]}=\frac{k-1}{k}\sum_{u=1}^k\sum_{i=1}^q\nu^{[1]}_{(f^u,H_i)}.
\end{align*}
Then, by Lemma \ref{lem3.2} we have
$$q\le 2N-n+1+\frac{kn(2N-n+1)}{(k-1)N(n+1)}+\frac{kn^2\rho}{N}.$$
This is a contradiction. 

Then $f^1\wedge\cdots\wedge f^k\equiv 0$. The theorem is proved.
\end{proof}

\section{Proof of  Theorem \ref{thm3}} 

Since the case where $M=\C^m$ have already proved by the author in \cite{Q14}, without loss of generality, in this proof we only consider the case where $M=\mathbb B^m(1)$. 

We now define:

$\bullet$ $F_k^{ij}=\dfrac {(f^k,H_i)}{(f^k,H_j)}\ (0\le k\le 2,\ 1\le i,j\le 2n+2),$

$\bullet$ $V_i=( (f^1,H_i),(f^2,H_i),(f^3,H_i))\in\mathcal M_m^3,$

$\bullet$ $\nu_i$: the divisor whose support is the closure of the set 
$$\{z;\nu_{(f^u,H_i)}(z)\ge \nu_{(f^v,H_i)}(z)=\nu_{(f^t,H_i)}(z)\text{ for a permutation } (u,v,t) \text{ of } (1,2,3)\}.$$

We write $V_i\cong V_j$ if $V_i\wedge V_j\equiv 0$, otherwise we write $V_i\not\cong V_j.$ For $V_i\not\cong V_j$, we write $V_i\sim V_j$ if there exist $1\le u<v\le 3$ such that $F_u^{ij}=F_v^{ij}$, otherwise we write $V_i\not\sim V_j$.

The following lemma is an extension of  \cite[Proposition 3.5]{Fu98} to the case of K\"{a}hler manifolds.
\begin{lemma}\label{lem4.1}
With the assumption of Theorem \ref{thm3}, let $f^1,f^2,f^3$ be three meromorphic mappings in $\mathcal D(f,\{H_i\}_{i=1}^{q},1)$.
Assume that there exist $i\in\{1,\ldots,q\}$, $c\in\mathcal C$ and  $\alpha\in\N^m$ with $|\alpha|=1$ such that
$\Phi^{\alpha}_{ic}\not\equiv 0.$
Then there exists a holomophic function $g_{i}\in B(1,1;f^1,f^2,f^3)$ such that
\begin{align*}
\nu_{g_{i}}\ge \nu^{[1]}_{(f,H_i)}+2\sum_{\underset{j\ne i}{j=1}}^q\nu^{[1]}_{(f,H_j)}
\end{align*}
\end{lemma}
\begin{proof}
We have
\begin{align}\label{eq4.2}
\begin{split}
\Phi^{\alpha}_{ic}
&=F_1^{ic}\cdot F_2^{ic}\cdot F_3^{ic}\cdot 
\left | 
\begin {array}{cccc}
1&1&1\\
F_1^{ci}&F_2^{ci} &F_3^{ci}\\
\mathcal {D}^{\alpha}(F_1^{ci}) &\mathcal {D}^{\alpha}(F_2^{ci}) &\mathcal {D}^{\alpha}(F_3^{ci})\\
\end {array}
\right|\\
&=\left | 
\begin {array}{cccc}
F_1^{ic}&F_2^{ic} &F_3^{ic}\\
1&1&1\\
F_1^{ic}\mathcal {D}^{\alpha}(F_2^{ci}) &F_2^{ic}\mathcal {D}^{\alpha}(F_2^{ci}) &F_3^{ic}\mathcal {D}^{\alpha}(F_3^{ci})
\end {array}
\right|\\
&=F_1^{ic}\bigl(\dfrac{\mathcal {D}^{\alpha}(F_3^{ci})}{F^{ci}_3}-\dfrac{\mathcal {D}^{\alpha}(F_2^{ci})}{F^{ci}_2}\bigl)
+F^{ic}_2\bigl(\frac{\mathcal {D}^{\alpha}(F_1^{ci})}{F^{ci}_1}-\frac{\mathcal {D}^{\alpha}(F_3^{ci})}{F^{ci}_3}\bigl)\\
&\ \ +F^{ic}_3\bigl(\frac{\mathcal {D}^{\alpha}(F_2^{ci})}{F^{ci}_2}-\frac{\mathcal {D}^{\alpha}(F_1^{ci})}{F^{ci}_1}\bigl).
\end{split}
\end{align}
This implies that
$$(\prod_{u=1}^3(f^u,H_c))\cdot\Phi^{\alpha}_{ic}=g_{i},$$
where
\begin{align*}
g_{i}=&(f^1,H_i)\cdot (f^2,H_c)\cdot (f^3,H_c)\cdot\left (\dfrac{\mathcal {D}^{\alpha}(F_3^{ci})}{F^{ci}_3}-\dfrac{\mathcal {D}^{\alpha}(F_2^{ci})}{F^{ci}_2}\right)\\
&+(f^1,H_c)\cdot (f^2,H_i)\cdot (f^3,H_c)\cdot\left (\dfrac{\mathcal {D}^{\alpha}(F_1^{ci})}{F^{ci}_1}-\dfrac{\mathcal {D}^{\alpha}(F_3^{ci})}{F^{ci}_3}\right)\\
&+(f^1,H_c)\cdot (f^2,H_c)\cdot (f^3,H_i)\cdot\left (\dfrac{\mathcal {D}^{\alpha}(F_2^{ci})}{F^{ci}_2}-\dfrac{\mathcal {D}^{\alpha}(F_1^{ci})}{F^{ci}_1}\right).
\end{align*}
Hence, we easily see that
$$|g_{i}|\le C\cdot \|f^1\|\cdot \|f^2\|\cdot \|f^3\|\cdot\sum_{u=1}^3\left|\dfrac{\mathcal {D}^{\alpha}(F_u^{ci})}{F^{ci}_u}\right|,$$
where $C$ is a positive constant, and then $g_{i}\in B(1;1;f^1,f^2,f^3)$.
It is clear that
\begin{align}\label{eq4.3}
\nu_{g_{i}}=\nu_{\Phi^{\alpha}_{ic}}+\sum_{u=1}^3\nu_{(f^u,H_c)}.
\end{align}

It is clear that $g_i$ is holomorphic on a neighborhood of each point of $\bigcup_{u=1}^3(f^u,H_c)^{-1}\{0\}$ which is not contained in  $\bigcup_{i=1}^q(f,H_i)^{-1}\{0\}$. Hence, we see that all zeros and poles of $g_i$ are points contained in some analytic sets $(f,H_s)^{-1}\{0\}\ (1\le s\le q)$. We note that the intersection of any two of these set has codimension at least two. It is enough for us to prove that (\ref{eq4.3}) holds for each regular point $z$ of the analytic set $\bigcup_{i=1}^q(f,H_i)^{-1}\{0\}.$ We distinguish the following cases:

\vskip0.1cm
\noindent
\textit{Case 1:}  $z\in\supp\nu_{(f,Hj)}\ (j\ne i)$. We write $\Phi^\alpha_{ic}$ in the form
$$ \Phi^{\alpha}_{ic}
=F_1^{ic}\cdot F_2^{ic}\cdot F_3^{ic}
\times\left | 
\begin {array}{cccc}
 \bigl (F_1^{ci}-F_{2}^{ci}\bigl ) & \bigl (F_1^{ci}-F_{3}^{ci}\bigl )\\
\mathcal{D}^{\alpha}\bigl ( F_1^{ci}-F_{2}^{ci}\bigl ) & \mathcal{D}^{\alpha}\bigl ( F_1^{ci}-F_{3}^{ci}\bigl )
\end {array}
\right|. $$
Then  by the assumption that $f^1,f^2,f^3$ coincide on $\supp\nu_{(f,Hj)}$, we have $F_1^{ci}=F_2^{ci}=F_3^{ci}$ on $\supp\nu_{(f,Hj)}$. The property of the general Wronskian implies that 
$$\nu_{\Phi^{\alpha}_{ic}}(z)\ge 2=\nu^{[1]}_{(f,H_i)}(z)+2\sum_{\underset{j\ne i}{i=1}}^{q}\nu^{[1]}_{(f,H_i)}(z).$$
 
\textit{Case 2:} $z\in\supp\nu_{(f,H_i)}$. 

\textit{Subcase 2.1:}
We may assume that $2\le \nu_{(f^1,H_{i})}(z)\le\nu_{(f^2,H_{i})}(z)\le\nu_{(f^3,H_{i})}(z)$. We write
\begin{align*}
\Phi^{\alpha}_{ic}
=F_1^{ic}\biggl [F_2^{ic}(F_1^{ci}-F_{2}^{ci})F_3^{ic}\mathcal{D}^{\alpha} (F_1^{ci}-F_{3}^{ci})-F_3^{ic}(F_1^{ci}-F_{3}^{ci})F_2^{ic}\mathcal{D}^{\alpha} (F_1^{ci}-F_{2}^{ci})\biggl ]
\end{align*}
It is easy to see that $F_2^{ic}(F_1^{ci}-F_{2}^{ci})$, $F_3^{ic}(F_1^{ci}-F_{3}^{ci})$ are holomorphic on a neighborhood of $z$, and 
\begin{align*}
&\nu^{\infty}_{F_3^{ic}\mathcal{D}^{\alpha}(F_1^{ci}-F_{3}^{ci})}(z)\le 1, \\
\text{ and }\ \ \ &\nu^{\infty}_{F_2^{ic}\mathcal{D}^{\alpha} (F_1^{ci}-F_{2}^{ci})}(z)\le 1.
\end{align*}
Therefore, it implies that
\begin{align*}
\nu_{\Phi^{\alpha}_{ic}}(z)&\ge 1 =\nu^{[1]}_{(f,H_i)}(z)+2\sum_{\underset{j\ne i}{i=1}}^{q}\nu^{[1]}_{(f,H_j)}(z) .
\end{align*}

\textit{Subcase 2.2:}
We may assume that $\nu_{(f^1,H_{i})}(z)=\nu_{(f^2,H_{i})}(z)=\nu_{(f^3,H_{i})}(z)=1$. We choose a neighborhood $U$ of $z$ and a holomorphic function $h$ without multiple zero on $U$ such that $\nu_h=\nu_{(f^u,H_{i})}\ (1\le u\le 3)$ on $U$. Hence $F^{ic}_u=hG^{ic}_u$ for non-vanishing holomorphic functions $G^{ic}_u$ on $U$. By the properties of Wronskian, we have $\Phi^{\alpha}_{ic}=h\Phi(G^{ic}_1,G^{ic}_2,G^{ic}_3)$ on $U$. This implies that
\begin{align*}
\nu_{\Phi^{\alpha}_{ic}}(z)=\nu_h(z)=\nu^{[1]}_{(f,H_i)}(z)+2\sum_{\underset{j\ne i}{i=1}}^{q}\nu^{[1]}_{(f,H_j)}(z) .
\end{align*}

From the above three cases, the inequality (\ref{eq4.3}) holds. The lemma is proved.
\end{proof}

\begin{proof}[{\sc Proof of theorem \ref{thm3}}]  Denote by $P$ the set of all $i\in\{1,\ldots ,q\}$ satisfying there exist $c\in\mathcal C$, $\alpha\in\N^m$ with $|\alpha|=1$ such that $\Phi^{\alpha}_{ij}\not\equiv 0$. 

If $\sharp P\ge 3$, for instance we suppose that $1,2,3\in P$, then there exist three corresponding holomorphic functions $g_{1},g_{2},g_{3}$ as in Lemma \ref{lem4.1}. We have $g_1g_2g_3\in B(3,3;f^1,f^2,f^3)$ and
$$ \nu_{g_1g_2g_3}\ge 2\sum_{u=1}^3\sum_{i=1}^q\nu^{[1]}_{(f^u,H_i)}-\frac{1}{3}\sum_{u=1}^3\sum_{i=1}^3\nu^{[1]}_{(f^u,H_i)}\ge \frac{5}{3}\sum_{u=1}^3\sum_{i=1}^q\nu^{[1]}_{(f^u,H_i)}.$$
Then, by Theorem \ref{lem3.2} we have
$$q\le (2N-n+1)\left (1+\frac{9n}{5N(n+1)}\right)+\rho\left (3n+\dfrac{9n}{5N}\right ).$$
This is a contradiction.

Hence $\sharp P\le 2$. We suppose that $i\not\in P\ \forall i=1,\ldots,q-2.$ Therefore, for all $i\in\{1,\ldots,q-2\}$ and $\alpha\in\N^m$ with $|\alpha|=1$ we have 
$$\Phi^{\alpha}_{ic}\equiv 0\ \forall c\in\mathcal C.$$
By the density of $\mathcal C$ in $\C^{n+1}$, the above identification holds for all $c\in\C^{n+1}\setminus\{0\}$. 

In particular, $\Phi^{\alpha}_{ij}\equiv 0$ for all $i\in\{1,\ldots,q-2\}$ and $\alpha\in\N^m$. Then for $1\le i<j\le q-2$, one of two following assertions holds:

(i) $F^{ij}_1=F^{ij}_2$ or $F^{ij}_2=F^{ij}_3$ or $F^{ij}_3=F^{ij}_1$.

(ii) $\dfrac{F^{ij}_1}{F^{ij}_2}$, $\dfrac{F^{ij}_2}{F^{ij}_3}$ and $\dfrac{F^{ij}_3}{F^{ij}_1}$ are all constant.

\begin{claim}\label{cl4.4}
For any two indices $i,j$ if there exists two mappings of  $\{f^1,f^2,f^3\}$, for instance they are $f^1,f^2$ such that $F^{ij}_1=F^{ij}_2$ then $F^{ij}_1=F^{ij}_2=F^{ij}_3$.
\end{claim}

Indeed, suppose contrarily that $F^{ij}_1=F^{ij}_2\ne F^{ij}_3$. Denote by $\mathcal M$ the field of all meromorphic functions on $\B^m(1)$. Then two vectors 
$$\left (\frac{(f^1,H_i)}{(f^1,H_j)},\frac{(f^2,H_i)}{(f^2,H_j)},\frac{(f^3,H_i)}{(f^3,H_j)}\right ) \text{ and } \left (\frac{(f^1,H_j)}{(f^1,H_j)},\frac{(f^2,H_j)}{(f^2,H_j)},\frac{(f^3,H_j)}{(f^3,H_j)}\right )$$
 are linear independent on $\mathcal M$. Since $f^1\wedge f^2\wedge f^3\equiv 0$, the vector 
$$\left (\frac{(f^1,H_s)}{(f^1,H_j)},\frac{(f^2,H_s)}{(f^2,H_j)},\frac{(f^3,H_s)}{(f^3,H_j)}\right)$$ 
belongs to the vector space spanned by two above vectors on $\mathcal M$ for all $s$. Since $\frac{(f^1,H_i)}{(f^1,H_j)}=\frac{(f^2,H_i)}{(f^2,H_j)}$ and $\frac{(f^1,H_j)}{(f^1,H_j)}=\frac{(f^2,H_j)}{(f^2,H_j)}$, it yields that $\frac{(f^1,H_s)}{(f^1,H_j)}=\frac{(f^2,H_s)}{(f^2,H_j)}$ for all $1\le s\le q$. This implies that $f^1=f^2$, which contradicts to the supposition.
Hence, we must have $F^{ij}_1=F^{ij}_2=F^{ij}_3$. The claim is proved.

From the above claim we see that for any two indices $1\le i,j\le q-2$ and two mappings $f^u,f^v$ we must have $F^{ij}_u=F^{ij}_v$ or there exists a constant $\alpha\ne 1$ with $F^{ij}_u=\alpha F^{ij}_v$.

Now we suppose that, there exists $F^{ij}_1=\beta F^{ij}_2\ (i<j)$ with $\beta\ne 1$. Since $F^{ij}_1=F^{ij}_2$ on $\bigcup_{s\ne i,j}(f,H_i)^{-1}\{0\}$, it follows that $\bigcup_{s\ne i,j}(f,H_s)^{-1}\{0\}=\emptyset$. Take an index $t\in\{1,\ldots,q-2\}\setminus\{i,j\}$, then we must have $F^{it}_1\ne F^{it}_2$ or $F^{jt}_1\ne F^{jt}_2$. For instance, we suppose that $F^{it}_1\ne F^{it}_2$. Similarly as above, we have $\bigcup_{s\ne i,t}(f,H_s)^{-1}\{0\}=\emptyset$. Therefore $\bigcup_{s\ne i}(f,H_s)^{-1}\{0\}=\emptyset$. This implies that $\delta_{f}^{[1]}(H_s)=1$ for all $s\in\{1,\ldots,q\}\setminus\{i\}$. By Theorem \ref{lem3.2}, we have
$$ q-1\le 2N-n+1+\rho\frac{(2N-n+1)n}{N+1}.$$
This is a contradiction.

Therefore,  $F^{ij}_1=F^{ij}_2=F^{ij}_3$ for all $1\le i<j\le q-2$. This implies that $f^1=f^2=f^3$. The supposition is false.

Hence, we must have $f^1=f^2$ of $f^2=f^3$ or $f^3=f^1$.
The theorem is proved.
\end{proof}

\vskip0.2cm
\noindent
{\bf Disclosure statement:} The author states that there is no conflict of interest. 

\vskip0.2cm
\noindent
{\bf Data Availability Statements:} The datasets generated during and/or analysed during the current study are available in the documents in the reference list.

\end{document}